\documentclass[a4paper,12pt]{amsart}
\usepackage[utf8]{inputenc}
\usepackage{amsmath}
\usepackage{amsfonts}
\usepackage{amsthm}
\usepackage{amssymb}
\usepackage[mathscr]{euscript}
\usepackage{xcolor}
\usepackage{enumerate}
\usepackage{tikz}
\usepackage{amstext}
\usepackage{enumerate}
\usepackage{amsfonts}
\usepackage{textcomp}
\usepackage{amssymb}
\usepackage{setspace}
\usepackage[all]{xy}
\usepackage{color}

\usepackage[
margin=1in,
marginpar=2cm,
includefoot,
footskip=30pt,
]{geometry}

\newtheorem{theorem}[equation]{Theorem}
\newtheorem{lemma}[equation]{Lemma}
\newtheorem{proposition}[equation]{Proposition}
\newtheorem{corollary}[equation]{Corollary}

\theoremstyle{definition}
\newtheorem{definition}[equation]{Definition}

\theoremstyle{remark}

\newtheorem{remark}[equation]{Remark}

\numberwithin{equation}{section}

\newcommand{\osf}{{\normalfont \textsf{X}}}

\newcommand{\lang}{\CL_{\osf}}

\newcommand{\ualgshift}{\TCA_R(\osf)}

\newcommand{\udalgshift}{\TCD_R(\osf)}

\newcommand{\alf}{\mathscr{A}}

\newcommand{\N}{\mathbb{N}}

\newcommand{\F}{\mathbb{F}}
\newcommand{\PP}{\mathbb{P}}

\newcommand{\CA}{\mathcal{A}}

\newcommand{\CD}{\mathcal{D}}

\newcommand{\CL}{\mathcal{L}}

\newcommand{\TCA}{\widetilde{\CA}}
\newcommand{\TCB}{\mathcal{U}}
\newcommand{\TCD}{\widetilde{\CD}}

\newcommand{\nn}{\mathbb{N}}

\newcommand{\eword}{\omega}

\title{Irreducible representations of one-sided subshift algebras}

\author[D. Gonçalves]{Daniel Gonçalves}
\address[Daniel Gonçalves and Danilo Royer]{Departamento de Matem\'atica, Universidade Federal de Santa Catarina, 88040-970 Florian\'opolis SC, Brazil. }
\email{daemig@gmail.com \\ daniloroyer@gmail.com}
\author[D. Royer]{Danilo Royer}

\begin{document}


\thanks{The first author was partially supported by Capes-Print Brazil and CNPq - Conselho Nacional de Desenvolvimento Científico e Tecnológico - Brazil.}

\begin{abstract}
Given a subshift over an arbitrary alphabet, we construct a representation of the associated unital algebra. We describe a criteria for the faithfulness of this representation in terms of the existence of cycles with no exits. Subsequently, we focus on describing the irreducible components of this representation and characterize equivalence between such components. Building upon these findings, we identify the minimal left ideals of a subshift algebra associated with (what we call) line paths. 
\end{abstract}

\maketitle

\section{Introduction}

Motivated by deep connections with symbolic dynamics, algebras associated with a subshift over an arbitrary alphabet have recently been defined in \cite{BGGV}. These algebras generalize a large class of (ultragraph) Leavitt path algebras and conjugacy between subshifts is described via isomorphism of the unital algebras associated, see \cite{BGGV, BCGWLabel}. Driven by the extensively studied connections between the algebraic structure of the Leavitt path algebra associated with a graph, the combinatorial properties of the graph, and the dynamical properties of the associated subshift of finite type, our aim is to contribute to the understanding of the algebraic structure of the unital algebra associated with a subshift. Specifically, we focus on the description of irreducible representations and the identification of left minimal ideals of the unital subshift algebra.

Irreducible representations (or simple modules) of Leavitt path algebras have been studied exhaustively in the literature, see \cite{Namirred, Aranga, Chen, FGGirred, MR2785948, GRirred,  reduction, rangaR, Lia} for example.
One of the pivotal developments in the theory was presented in \cite{Chen}, where Chen considered a module with a basis consisting of infinite paths tail equivalent to a given one-sided infinite path $q$ in a graph $E$, and proved that this module is $L_K(E)$-simple. 
Chen also noticed that this module may be obtained via branching systems (a concept developed in \cite{MR2785948}) and studied minimal left ideals. In this paper, we generalize the construction in \cite{Chen} to unital subshift algebras and use it to identify minimal left ideals.

Minimal left ideals of Leavitt path algebras have been described in \cite{Socle, Socle1} in terms of line points. Using this description, Chen showed that the irreducible modules constructed in \cite{Chen} are in correspondence with the left minimal ideals of a Leavitt path algebra. In the context of subshift algebras, a comprehensive description of all left minimal ideals is lacking. Thus, we take an alternative approach and describe minimal left ideals of the subshift algebra using the irreducible modules we construct. 
 In a follow-up paper, we plan to provide a comprehensive description of all left minimal ideals of a subshift algebra, to enhance our understanding of its socle (the sum of all its minimal left ideals), which is an important algebraic concept, (see \cite{Socle, Socle1, Socle3} for the socle theory of Leavitt path algebras). 

This paper is organized as follows. After this brief introduction, Section 2 revisits the fundamental concepts regarding the unital subshift algebra. 
In Section 3, given a subshift $\osf$, we construct a representation of the associated unital algebra $\ualgshift$, where $R$ is a ring, in the space of endomorphisms of $\PP$ (see Proposition~\ref{representationtheorem}). Here, $\PP$ denotes the free $R$-module generated by elements in $\osf$.
We characterize the faithfulness of this representation in terms of cycles with no exits, as stated in Theorem \ref{cyclenoexit}. Subsequently, we delve into the study of irreducible components (submodules) of the representation presented in Proposition~\ref{representationtheorem}. We introduce the concept of tail equivalence (Definitition~\ref{jacarenapraia}) and demonstrate that the submodule generated by a tail equivalence relation is irreducible (provided $R$ is a field), see Theorem \ref{invariant submodules}. We conclude the section by characterizing the set of all $\ualgshift$-homomorphisms of the submodule generated by a tail equivalence relation (Proposition \ref{endo}), as well as determining when two such submodules are equivalent (Proposition~\ref{tinto}). In the final section,  we focus on the study of minimal left ideals of the unital algebra associated with a subshift. We introduce the concept of a line path (Definition~ \ref{linepath}) and show that it induces a left ideal of $\ualgshift$ isomorphic to the submodule generated by the tail equivalence relation of the line path (Theorem~\ref{eletrico}). By combining the results from Sections 3 and 4, we establish that the left ideal induced by a line path is minimal (Corollary~\ref{madruga}).

\section{Preliminaries}\label{eleicoes}

Throughout the paper, $R$ stands for a commutative unital ring, $\nn=\{0,1,2,\ldots\}$ and $\nn^*=\{1,2,\ldots\}$.

We start the preliminaries with the basic concepts of symbolic dynamics that will be necessary for our work (we follow the notation of \cite{BGGV}).

\subsection{Symbolic Dynamics}\label{symbolic}

Let $\alf$ be a non-empty set, called an \emph{alphabet}, and let $\sigma$ be the \emph{shift map} on $\alf^\N$, that is, $\sigma$ is the map from $\alf^\N$ to $\alf^\N$ given by $\sigma(x)=y$, where $x=(x_n)_{n\in \N}$ and  $y=(x_{n+1})_{n\in \N}$. Elements of $\alf^*:=\bigcup_{k=0}^\infty \alf^k$ are called \emph{blocks} or \emph{words}, and $\omega$ stands for the empty word. We also set $\alf^+=\alf^*\setminus\{\eword\}$. Given $\alpha\in\alf^*\cup\alf^{\N}$, $|\alpha|$ denotes the length of $\alpha$ and for $1\leq i,j\leq |\alpha|$, we define $\alpha_{i,j}:=\alpha_i\cdots\alpha_j$ if $i\leq j$, and $\alpha_{i,j}=\eword$ if $i>j$. If moreover $\beta\in\alf^*$, then $\beta\alpha$ denotes the usual concatenation. For a block $\alpha\in \alf^k$, $\alpha^\infty$ denotes the infinite word $\alpha \alpha \ldots$. A subset $\osf \subseteq \alf^\N$ is \emph{invariant} for $\sigma$ if $\sigma (\osf)\subseteq \osf$. For an invariant subset $\osf \subseteq \alf^\N$, we define $\CL_n(\osf)$ as the set of all words of length $n$ that appear in some sequence of $\osf$, that is, $$\CL_n(\osf):=\{(a_0\ldots a_{n-1})\in \alf^n:\ \exists \ x\in \osf \text{ s.t. } (x_0\ldots x_{n-1})=(a_0\ldots a_{n-1})\}.$$ Clearly, $\CL_n(\alf^\N)=\alf^n$ and we always have that $\CL_0(\osf)=\{\omega\}$.
The \emph{language} of $\osf$ is the set $\lang$, which consists of all finite words that appear in some sequence of $\osf$, that is,
$$\lang:=\bigcup_{n=0}^\infty\CL_n(\osf).$$
For elements $c,d\in \lang$, we use the notation $c_{|c|}\neq d_{|d|}$ to mean that the last letter of $c$ is different from the last letter of $d$.

Given $F\subseteq \alf^*$, the \emph{subshift} $\osf_F\subseteq \alf^\N$ is the set of all sequences $x$ in $\alf^\N$ such that no word of $x$ belongs to $F$. When the context is clear, we denote $\osf_F$ by $\osf$. 
The key sets that are used in the definition of the algebra associated with a subshift are defined below. 

\begin{definition}\label{Arapaima}
Let  $\osf$ be a subshift for an alphabet $\alf$. Given $\alpha,\beta\in \lang$, define \[C(\alpha,\beta):=\{\beta x\in\osf:\alpha x\in\osf\}.\] In particular, the set  $C(\eword,\beta)$ is denoted by $Z_{\beta}$ and called a \emph{cylinder set}, and the set  $C(\alpha,\eword)$ is denoted by $F_{\alpha}$ and called a \emph{follower set}. Notice that $\osf=C(\eword,\eword)$. 
\end{definition}

\subsection{Unital algebras of subshifts}\label{unital}

In this section, we recall the definition of the unital algebra associated with a general subshift $\osf$, as done in \cite{BGGV}. We start with the definition of the Boolean algebra associated with the sets of the form $C(\alpha,\beta)$.

\begin{definition}\label{diachuvoso}
Let $\osf$ be a subshift. Define $\TCB$ to be the Boolean algebra of subsets of $\osf$ generated by all $C(\alpha,\beta)$ for $\alpha,\beta\in\lang$, that is, $\TCB$ is the collections of sets obtained from finite unions, finite intersections, and complements of the sets $C(\alpha,\beta)$.
\end{definition}

We can now recall the definition of the unital algebra associated with a subshift.

\begin{definition}\label{universal properties}
Let $\osf$ be a subshift. We define the \emph{unital subshift algebra} $\ualgshift$ as the universal unital $R$-algebra  with generators $\{p_A: A\in\TCB\}$ and $\{s_a,s_a^*: a\in\alf\}$, subject to the relations:
\begin{enumerate}[(i)]
    \item $p_{\osf}=1$, $p_{A\cap B}=p_Ap_B$, $p_{A\cup B}=p_A+p_B-p_{A\cap B}$ and $p_{\emptyset}=0$, for every $A,B\in\TCB$;
    \item $s_as_a^*s_a=s_a$ and $s_a^*s_as_a^*=s_a^*$ for all $a\in\alf$;
    \item $s_{\beta}s^*_{\alpha}s_{\alpha}s^*_{\beta}=p_{C(\alpha,\beta)}$ for all $\alpha,\beta\in\lang$, where $s_{\eword}:=1$ and, for $\alpha=\alpha_1\ldots\alpha_n\in\lang$, $s_\alpha:=s_{\alpha_1}\cdots s_{\alpha_n}$ and $s_\alpha^*:=s_{\alpha_n}^*\cdots s_{\alpha_1}^*$.
\end{enumerate}
\end{definition}

Recall that the relative range of $(A,\alpha)$, where $\alpha\in \lang$ and $A\in \TCB$, is given by \[r(A,\alpha)=\{x\in\osf: \alpha x\in A\}.\]

An important feature of $\ualgshift$ is its grading by the free group on the alphabet. This grading arises from the isomorphism of the algebra with a certain partial skew group ring, a result proved in \cite{BGGV} that we recall below.

\begin{theorem}\label{thm:set-theoretic-partial-action}
Let $\osf$ be a subshift. Then, $\ualgshift\cong\udalgshift\rtimes_{\tau}\F$ via an isomorphism $\Phi$ that sends $s_a$ to $1_a\delta_a$ and $s^*_a$ to $1_{a^{-1}}\delta_{a^{-1}}$.
\end{theorem}

\section{Representations}

\begin{definition} Let $\osf$ be a subshift and $R$ be a ring. We denote by $\PP$ the free $R$-module generated by the the basis $\{\delta_x:x\in \osf\}$. 
\end{definition}

In what follows, we use the notation $[p]$ meaning that $[p]=1$ if the statement $p$ is true and $[p]=0$ otherwise.

Define, for all $A\in \TCB$, the $R$- endomorphism $P_A:\PP\rightarrow \PP$ in a basis element $\delta_x$, $x\in \osf$, by 
$$P_A(\delta_x)=[x\in A]\delta_x,$$ and, for each $a\in \alf$, define the $R$-endomorphisms $S_a,S_a^*:\PP\rightarrow \PP$ in a basis element $\delta_x$, $x=x_1x_2x_3...\in \osf$, by 
$$S_a(\delta_x)=[ax\in \osf]\delta_{ax}$$ and 
$$S_a^*(\delta_x)=[x_1=a]\delta_{\sigma(x)}.$$

Denote the $R$-algebra of all the $R$-endomophisms of $\PP$ by $End_R(\PP)$. Clearly, $P_A,S_a,S_a^*\in End_R(\PP)$, for each $A\in \TCB$ and $a\in \alf$.

We show next that the endomorphisms defined above form a representation of $\ualgshift$.

\begin{proposition}\label{representationtheorem}
Let $\osf$ be a subshift and $R$ be a ring. Then there exists an $R$-homomorphism $\pi:\ualgshift\rightarrow End_R(\PP)$ such that $$\pi(p_A)=P_A, \pi(s_a)=S_A \text{ and } \pi(s_a^*)=S_a^*$$ for each $A\in \TCB$ and $a\in \alf$.
\end{proposition}

\begin{proof} Since $\ualgshift$ is a universal algebra, the results follow once we check that the elements $P_A, S_a, S_a^*$, with $A\in \TCB$ and $a\in \alf$, satisfy the three items of Definition~\ref{universal properties}. 

The first item follows from direct calculations and is left to the reader.

For the second item, let $a\in \alf$ and $x\in \osf$. If $ax\in X$ then $$S_aS_a^*S_a(\delta_x)=S_aS_a^*(\delta_{ax})=S_a(\delta_x),$$ and if $ax\notin \osf$ then $S_a(\delta_x)=0$. Therefore, $S_aS_a^*S_a(\delta_x)=S_a(\delta_x)$ for each $x\in \osf$ and, consequently, $S_aS_a^*S_a=S_a$ for each $a\in \alf$. Similarly, we have that $S_a^*S_aS_a^*=S_a^*$ for each $a\in \alf$.

To prove the third item, let $\alpha,\beta \in \lang$, and $x\in \osf$. If $x\in C(\alpha,\beta)$, then $x=\beta y$ and $\alpha y\in \osf$.
Therefore, in this case, we have that
$$S_\beta S_\alpha^* S_\alpha S_\beta^*(\delta_x)=S_\beta S_\alpha^* S_\alpha S_\beta^*(\delta_{\beta y})=S_\beta S_\alpha^* S_\alpha (\delta_y)=$$
$$=S_\beta S_\alpha^*(\delta_{\alpha y})=S_\beta (\delta_y)=\delta_{\beta y}=\delta_x=P_{C(\alpha,\beta)}(\delta_x).$$
Suppose now that $x=x_1x_2x_3...\notin C(\alpha,\beta)$. Then $x_1...x_{|\beta|}\neq \beta$, or $x=\beta y
$ and $\alpha y\notin \osf$.
If $x_1...x_{|\beta|}\neq \beta$ then $S_\beta^*(\delta_x)=0$. If $x=\beta y$ and $\alpha y\notin \osf$ then $$S_\alpha S_\beta^*(\delta_x)=S_\alpha S_\beta^*(\delta_{\beta y})=S_\alpha (\delta_{y})=0.$$ 
Therefore, when $x\notin C(\alpha,\beta)$, we have that
$$S_\beta S_\alpha^* S_\alpha S_\beta^*(\delta_x)=0=p_{C(\alpha,\beta)}(\delta_x).$$ 
So, we have proved that for every $x\in \osf$ the equality 
$S_\beta S_\alpha^* S_\alpha S_\beta^*(\delta_x)=P_{C(\alpha,\beta)}(\delta_x)$ is true, and consequently $$S_\beta S_\alpha^* S_\alpha S_\beta^*=P_{C(\alpha,\beta)},$$ and so the third item of Definition~\ref{universal properties} is proved.
\end{proof}

\begin{remark} The homomorphism $\pi$ of Proposition~\ref{representationtheorem} is not always faithful. As an example, let $\alf=\{a\}$ and $\osf$ be the singleton set $\osf=\{a^\infty\}$. Then $S_a(\delta_{a^\infty})=\delta_{a^\infty}=P_X(\delta_{a^\infty})$ and hence $S_a=P_{\osf}$ (in $End_R(\PP)$). However, using Theorem~\ref{thm:set-theoretic-partial-action} we obtain that $s_a\neq p_{\osf}$ and so the homomorphism $\pi$ of Proposition~\ref{representationtheorem} is not faithful. 
\end{remark}

In our next result, we characterize when the homomorphism $\pi$ of Proposition~\ref{representationtheorem} is faithful. For this, we need to recall the notion of a cycle without exit defined in  \cite{reductiontheoremofsubshifts}.

\begin{definition}
    Let $\osf$ be a subshift, $c\in \lang$, and $\emptyset \neq A\in \TCB$. The pair $(c,\alpha)$ is a \emph{cycle} if $A\subseteq r(A,c)$. The cycle $(A,c)$ has no exit if $A=\{c^\infty\}$.
\end{definition}

\begin{theorem}\label{cyclenoexit} let $\osf$ be a subshift and $R$ be a ring. Then, the homomorphism $\pi:\ualgshift\rightarrow End_R(\PP)$ is faithful if and only if $\osf$ has no cycle without exit. 
\end{theorem}

\begin{proof} Suppose that there is no cycle without exit. Then, since $\pi(\gamma p_A)=\gamma P_A\neq 0$ for each $0\neq \gamma \in R$, it follows from the Cuntz-Krieger Uniqueness Theorem for subshift algebras (see \cite{reductiontheoremofsubshifts}) that $\pi$ is faithful.

For the converse, suppose that there exists a cycle without exit $(A,c)$ in $\osf$. Then, for each $x\in \osf$, we have that
$$S_c P_A(\delta_x)=[x\in A]S_c(\delta_x)=[x\in A]S_c(\delta_{c^\infty})=[x\in A]\delta_{c^\infty}=[x\in A]p_A(\delta_x),$$
and therefore $S_cP_A=P_A$. So, $\pi(s_cp_A)=\pi(p_A)$ but, from the grading given by Theorem~\ref{thm:set-theoretic-partial-action} we obtain that $s_cp_A\neq p_A$, and hence $\pi$ is not faithful.
\end{proof}

\begin{remark}
The homomorphism $\pi$ of Proposition~\ref{representationtheorem} is usually not irreducible. For instance, let $\osf$ be a subshift that contains an element $y$ and an element $z$ such that $\sigma^n(y)\neq \sigma^m(z)$ for all $n,m \in \N$. Let $Y\subseteq \osf$ be the set
$$Y=\{x\in \osf: \sigma^n(x)=\sigma^m(z) \text{ for some }n,m\in \N\}$$
and let $M$ be the $R$-submodule of $\PP$ generated by $\{\delta_x:x\in Y\}$. It follows from the definition of $P_A, S_a$, and $S_a^*$ that $P_A(M)\subseteq M$, $S_A(M)\subseteq M$, and $S_a^*(M)\subseteq M$, for each $A\in \TCB$ and $a\in \alf$. Therefore, $M$ is an invariant submodule of $\pi$ (in the sense that $\pi(\ualgshift)(M)\subseteq M$). 
\end{remark}

The set $Y$ hints at how to build irreducible components of the representation $\pi$ and motivates the following definition.

\begin{definition}\label{jacarenapraia} Let $\osf$ be a subshift. We say that two elements $x,y\in \osf$ are equivalent, and write $x\thicksim y$, if $\sigma^n(x)=\sigma^m(y)$ for some $n,m\in\N$. 
\end{definition}

\begin{remark}\label{equiv1} Notice that $x\thicksim y$ if, and only if, there exists $c,d\in \lang$ and $\xi\in \osf$ such that $x=c\xi$ and $y=d\xi$. 
\end{remark}

The relation $\thicksim$ of Definition~\ref{jacarenapraia} is an equivalence relation. For an element $x\in \osf$, we denote its equivalence class by $[x]$ and define $\PP_{[x]}$ as the $R$-submodule of $\PP$ generated by $\{\delta_y:y\in [x]\}$. Let $\widetilde{\osf}=\{[x]:x\in X\}$ and observe that
$$\PP=\bigoplus\limits_{[x]\in \widetilde{\osf}}\PP_{[x]}.$$ 

We show below that each submodule $\PP_{[x]}$ is an irreducible component of the representation $\pi$. 

\begin{theorem}\label{invariant submodules}
    Let $\osf$ be a subshift, $R$ be a ring, and let $\pi:\ualgshift\rightarrow End_R(\PP)$ be the homomorphism of Proposition~\ref{representationtheorem}. 
    \begin{enumerate}
        \item For each $x\in \osf$, the submodule $\PP_{[x]}$ is $\pi$-invariant, i.e. $\pi(\ualgshift)(\PP_{[x]})\subseteq \PP_{[x]}$;
        \item If $R$ is a field then, for each $x\in \osf$, the submodule $\PP_{[x]}$ is irreducible, in the sense that there is no submodule $0\neq Y\varsubsetneq \PP_{[x]}$ such that $\pi(\ualgshift)(Y)\subseteq Y$.
    \end{enumerate}
\end{theorem}

\begin{proof} We begin with the proof of the first item. For this, fix an $x\in \osf$ and let $y\in \osf$ be such that $y\thicksim x$. Then, for $a\in \alf$, we have that $S_a(\delta_y)=[ay\in \osf]\delta_{ay}$. So, $S_a(\delta_y)=0$ except in case $ay\in \osf$, which implies that $ay\thicksim x$ (since $ay\thicksim y$ and $y\thicksim x$). Therefore, $\pi(s_a)(\PP_{[x]})\subseteq \PP_{[x]}$. Similarly, $\pi(s_a^*)(\PP_{[x]})\subseteq \PP_{[x]}$ and $\pi(p_A)(\PP_{[x]})\subseteq \PP_{[x]}$ for each $a\in \alf$ and $A\in \TCB$. Since $\pi$ is a homomorphism, the result follows. 

We proceed to the proof of the second item. Let $x\in \osf$ and let $0\neq Y\subseteq \PP_{[x]}$ be a $\pi$-invariant submodule. We show that $Y=\PP_{[x]}$. Let $0\neq y\in Y$. Then $y=\sum\limits_{i=1}^n \lambda_i \delta_{x^i}$, where $0\neq \lambda_i\in R$ and
$x^i\thicksim x$ for each $i$, and $x^i\neq x^j$ for each $i\neq j$. Write each $x^i$ as $x^i=x_1^ix_2^ix_3^i...$, $i\in \{1,...,n\}$.
Notice that, since $x^i\neq x^j$ for each $i\neq j$, there exists an $m\in \N$ such that 
$$x_1^ix_2^i...x_m^i\neq x_1^jx_2^j...x_m^j$$ for each $i,j\in \{1,...,n\}$ with $i\neq j$. 

Let $\mu=x_1^1x_2^1...x_m^1$. Then, $$\pi(\lambda_1^{-1}p_{Z_\mu})(y)=\delta_{x^1}$$ and, since $Y$ is $\pi$-invariant, we conclude that $\delta_{x^1}\in Y$.

To finish, we show that $\delta_z\in Y$ for each $z\in [x]$. Let $z\in [x]$. Then $z\thicksim x^1$ 
 and, from Remark~\ref{equiv1}, there exists $c,d\in \lang$ and $\xi\in \osf$ such that $x^1=c\xi$ and $z=d\xi$. Therefore, $$\pi(s_ds_c^*)(\delta_{x^1})=S_dS_c^*(\delta_{x^1})=S_dS_c^*(\delta_{c\xi})=S_d(\delta_{\xi})=\delta_{d\xi}=\delta_z$$
and, since $\delta_{x^1}\in Y$ and $Y$ is $\pi$-invariant, we conclude that $\delta_z\in Y$. Hence, $\PP_{[x]}\subseteq Y$ as desired. 
\end{proof}

We proved above that, for each $x\in \osf$, the $R$-submodule $\PP_{[x]}$ is $\pi$-invariant. Therefore, $\PP_{[x]}$ is a left $\ualgshift$-module, with the left product defined by $\mu.z=\pi(\mu)(z)$, for each $\mu \in \ualgshift$ and $z\in \PP_{[x]}$. Let $End_{\ualgshift}(\PP_{[x]})$ denote the set of all $\ualgshift$-homomorphisms of $\PP_{[x]}$. In the next proposition, we characterize this set.

\begin{proposition}\label{endo} Let $\osf$ be a subshift and $R$ be a ring. Then, $End_{\ualgshift}(\PP_{[x]})$ is isomorphic to $R$.    
\end{proposition}

\begin{proof} Let $\varphi\in End_{\ualgshift}(\PP_{[x]})$. For $y,z\in [x]$, let $c,d\in \lang$ and $\xi\in \osf$ as in Remark \ref{equiv1}, so that $y=c\xi$ and $z=d\xi$. Then, since $\varphi$ is an $\ualgshift$-homomorphism, we have that
$$\varphi(\delta_y)=\varphi(\delta_{c\xi})=\varphi(S_cS_d^*(\delta_{d\xi}))=\varphi(S_cS_d^*(\delta_z))=S_cS_d^*\varphi(\delta_z).$$ Therefore, $\varphi=0$ if $\varphi(\delta_z)=0$ for some $z\in [x]$.

Suppose that $\varphi(\delta_z)\neq 0$ for each $z\in [x]$. For $z\in [x]$, write $\varphi(\delta_z)=\sum\limits_{i=1}^n\gamma_i \delta_{x^i}$, where $x^i\in [x]$, $0\neq \gamma_i\in R$, and $x^i\neq x^j$ for all $i\neq j$. Suppose that $x^k\neq z$ for some $k$. Let $m\in \N$ be such that $x_1^kx_2^k...x_m^k\neq z_1z_2...z_m$ and $x_1^kx_2^k...x_m^k\neq x_1^ix_2^i...x_m^i$, for each $i\neq k$. 
Let $\alpha=x_1^kx_2^k...x_m^k$. Then,
$$\gamma_k\delta_{x^k}=P_{Z_\alpha}(\gamma_k\delta_{x^k})=P_{Z_\alpha}(\sum\limits_{i=1}^n \gamma_i \delta_{x^i})=P_{Z_\alpha}(\varphi(\delta_z))=\varphi(P_{Z_\alpha}(\delta_z))=0,$$
where the second to last equality is true because $\varphi$ is $\ualgshift$-invariant. We conclude that $\delta_{x^k}=0$, which is a contradiction.  Therefore,  for all $z\in [x]$, we have that $\varphi(\delta_z)=\gamma_z\delta_z$ for some $\gamma_z\in R$. Next, we show that the map $z\rightarrow \gamma_z$ is constant. Let $z\in [x]$. Write $x=c\xi$ and $z=d\xi$, where $c,d\in \lang$ and $\xi\in \osf$ (see Remark~\ref{equiv1}).
Then,
$$\gamma_x\delta_x=\varphi(\delta_x)=\varphi(S_cS_d^*(\delta_z))=S_cS_d^*(\varphi(\delta_z))=\gamma_zS_cS_d^*(\delta_z)=\gamma_z\delta_x,$$
and hence $\gamma_z=\gamma_x$ for each $z\in [x]$. We conclude that $\varphi(\mu)=\gamma_x \mu$ for each $\mu\in \ualgshift$. Denote the element $\gamma_x$ by $\gamma_\varphi$. Then,  $End_{\ualgshift}(\PP_{[x]})$ and $R$ are isomorphic via the map $$End_{\ualgshift}(\PP_{[x]})\ni \varphi\mapsto \lambda_\varphi \in R,$$ and the proof is finished.
\end{proof}

Next, we will describe when the irreducible representations of $\ualgshift$ induced by the modules $\PP_{[x]}$, $x\in \osf$, are equivalent. Before, we recall the definition of equivalent homomorphisms.

\begin{definition}\label{equivrepres} Let $\osf$ be a subshift, $R$ be a ring, and $M$ and $N$ be two $R$-modules. We say that the $R$-homomorphisms $\varphi: \ualgshift \rightarrow End_R(M)$ and $\phi: \ualgshift \rightarrow End_R(N)$ are equivalent if there exists an $R$-module isomorphism $U:M\rightarrow N$ such that, for each $\mu\in \ualgshift$, the diagram below commutes.
\begin{displaymath}
\xymatrix{
M \ar[r]^{\varphi(\mu)} \ar[d]_{U} &
M \ar[d]^{U} \\
N \ar[r]_{\phi(\mu)} & N }
\end{displaymath}
\end{definition}

Let $\pi:\ualgshift\rightarrow End_R(\PP)$ be the homomorphism of Proposition~\ref{representationtheorem}. By Theorem~\ref{invariant submodules}, for each $x\in \osf$, the submodule $\PP_{[x]}$ is $\pi$-invariant and hence the restriction of $\pi$ to $\PP_{[x]}$ is a homomorphism, which we still denote by $\pi$. In this manner, we obtain a new homomorphism $\pi:\ualgshift\rightarrow End_R(\PP_{[x]})$, which is irreducible (in case $R$ is a field, see Theorem~\ref{invariant submodules}). Below we describe when such homomorphisms are equivalent.

\begin{proposition}\label{tinto} Let $\osf$ be a subshift, $R$ be a ring, and $x,y\in \osf$. Then, the homomorphisms $\pi:\ualgshift\rightarrow End_R(\PP_{[x]})$ and $\pi:\ualgshift\rightarrow End_R(\PP_{[y]})$ are equivalent if, and only if, $[x]=[y]$.
\end{proposition}

\begin{proof} Let $x,y\in \osf$. Clearly, the two homomorphisms are equivalent if  $[x]=[y]$. 

Suppose that $[x]\neq [y]$. We show that the two homomorphisms are not equivalent. Suppose otherwise, that is, suppose that there exists an isomorphism $U:\PP_{[x]}\rightarrow \PP_{[y]}$ such that $U\circ \pi(\mu)=\pi(\mu)\circ U$ for each $\mu \in \ualgshift$. Notice that $U(\delta_x)=\sum\limits_{i=1}^n\gamma_i \delta_{y^i}$. Write $x=x_1x_2x_3...$ and $y^i=y_1^iy_2^iy_3^i...$, where $x_j, y_j^i\in \alf$. Since $[x]\neq [y]$ and $y^i\in [y]$ for each $i$, we have that $y^i\neq x$ for each $i\in \{1,...,n\}$. Therefore, there exists an $m\in \N$ such that $x_1x_2...x_m\neq y_1^iy_2^i...y_m^i$ for each $i\in \{1,...,n\}$. Let $\alpha=x_1x_2....x_m$. Then,
$$U(\delta_x)=U(P_{Z_\alpha}(\delta_x))=U(\pi(p_{Z_\alpha})(\delta_x))=\pi(p_{Z_\alpha})(U(\delta_x))=$$
$$=\pi(p_{Z_\alpha})(\sum\limits_{i=1}^n\gamma_i \delta_{y^i})=\sum\limits_{i=1}^n \gamma_i P_{Z_\alpha}(\delta_{y^i})=0.$$ 
Therefore, $U(\delta_x)=0$, which is a contradiction (since $U$ is an isomorphism). Hence, if $[x]\neq [y]$ then the two homomorphisms of the proposition are not equivalent, as desired. \end{proof}

\section{Minimal left ideals}

In this section, we show that the irreducible representations described before are isomorphic to minimal left ideals of $\ualgshift$. Recall that a left ideal $I$ of an algebra $A$ is said to be minimal if it is nonzero and the only left ideals of $A$ that it contains are $0$ and $I$. 
In the context of Leavitt path algebras, left minimal ideals correspond to line points (see \cite{Socle, Socle1}. We generalize this concept to subshifts below. 

\begin{definition}\label{linepath} Let $\osf$ be a subshift. We say that an element $q=a_0a_1a_2a_3...\in \osf$ is a line path if $Z_{a_{0}}=\{q\}$ and for every $\beta\in \lang$ and $k\in \N$ we have that $\beta^\infty\neq a_ka_{k+1}a_{k+2}\ldots$.
\end{definition}

Clearly, for each element $x\in \ualgshift$, the set $$\ualgshift x=\{yx:y\in \ualgshift\}$$ is a left ideal of $\ualgshift$. We are interested in left ideals induced by line paths, that is, in ideals of the form $\ualgshift p_{Z_a}$, where $q=aa_1a_2a_3...$ is a line path. We further characterize such ideals in the lemma below.

\begin{lemma}\label{linearspan} Let $\osf$ be a subshift and $q=aa_1a_2a_3...$ be a line path.  

\begin{enumerate}
\item If $c,d\in \lang$ and $A\in \TCB$ are such that $s_c p_A s_d^*p_{Z_a}\neq 0$ then,  $s_c p_A s_d^* p_{Z_a}=s_c s_d^* p_{Z_a}$.

\item The left ideal $\ualgshift p_{Z_a}$ is the linear span of the set
$$\{s_c s_d^* p_{Z_a}:c,d\in \lang \text{ and }c_{|c|}\neq d_{|d|}\}.$$
\end{enumerate}

\end{lemma}

\begin{proof} 
We begin with the proof of Item (1). 

First, observe that if $A,B\in \TCB$, and $c\in \lang$, are such that $B$ is a single-point set and $s_cp_A p_B\neq 0$, then $A\cap B = B$ and $s_c p_A p_{B}=s_c p_{B}$. Indeed, this follows from the fact that $p_A p_{B} = p_{A\cap B}$.


Let $c,d\in \lang$ and $A\in \TCB$ be as in the hypothesis. 
Recall, see  \cite[Lemma~3.3]{reductiontheoremofsubshifts}, that $s_d^*p_{Z_a}=p_{r(Z_a,d)}s_d^*$. Therefore, $$s_c p_A s_d^*p_{Z_a}=s_c p_A p_{r(Z_a, d)}s_d^*.$$ 
Notice that $s_c p_A p_{r(Z_a, d)}\neq 0$ and $r(Z_a,d)$ is a single-point set, since $Z_a$ is a single-point set. Therefore, by the observation above, we have that $A\cap r(Z_a, d)=r(Z_a,d)$ and 
$$s_c p_A p_{r(Z_a, d)}s_d^*=s_c p_{r(Z_a, d)}s_d^*=s_c s_d^* p_{Z_a}.$$ 
Hence, $s_c p_A s_d^* p_{Z_a}=s_c s_d^* p_{Z_a}$ and Item (1) is proved. 

Now we prove Item (2). Recall that $\ualgshift$ is the linear span of elements of the form $s_c p_A s_d^*$, where $c,d\in \lang$ and $A\in \TCB$. Therefore, $\ualgshift p_{Z_a}$ is the linear span of elements of the form $s_c p_A s_d^* p_{Z_a}$. By Item~(1), each  nonzero element $s_c p_A s_d^* p_{Z_a}$ can be written in the form $s_c s_d^* p_{Z_a}$, and hence $\ualgshift p_{Z_a}$ is the linear span of such elements.
Let $c,d\in \lang$ be such that $s_c s_d^* p_{Z_a}\neq 0$ and suppose that $c_{|c|}=d_{|d|}$. Write $c=\alpha e$ and $d=\beta e$, where $e$ is the last letter of $c$ and $d$. Then, 
$$s_c s_d^* p_{Z_a}=s_\alpha s_e s_e^* s_\beta^* p_{Z_a}=s_\alpha p_{Z_e} s_\beta^* p_{Z_a}=s_\alpha s_\beta^* p_{Z_a},$$ where the last equality follows from Item~(1). By applying repeatedly this argument, we get the desired result. 
\end{proof}

We will connect the left ideals defined above with the irreducible modules defined in the previous section. For this, we need the following definition. 

\begin{definition}
Let $\osf$ be a subshift, $q\in \osf$, and let $p\in [q]$. We say that a pair $(\beta,\alpha)\in \lang \times \lang$ is a reduced initial pair associated with $(p,q)$ if there exists $x\in \osf$ such that $p=\beta x$, $q=\alpha x$, and $\alpha_{|\alpha|}\neq \beta_{|\beta|}$. 
\end{definition} 

When $q$ is a line path and $p\in[q]$, the reduced initial pair associated with $(p,q)$ is unique, as we show below.

\begin{lemma} Let $\osf$ be a subshift, and $q\in \osf$ be a line path. Then, for each $p\in [q]$, there exists a unique reduced initial pair $(\beta,\alpha)$ associated with $(p,q)$. 
\end{lemma}

\begin{proof} Let $(\beta, \alpha)$ and $(\beta',\alpha')$ be reduced initial pairs associated with $(p,q)$. Then, there exists $x,x'\in \osf$ such that $\beta x=q=\beta' x'$ and $\alpha x=q=\alpha' x'$. To prove that $\alpha=\alpha'$ and $\beta = \beta'$ it is enough to show that $|\alpha|=|\alpha'|$ and $|\beta|=|\beta'|$. We divide the proof into a few cases.

Suppose first that $|\alpha'|>|\alpha|$ and $|\beta'|>|\beta|$. Since $\alpha x= \alpha'x'$, there exists $c\in \lang$ (with $|c|>0$) such that $\alpha'=\alpha c$. Hence, $\alpha c x'=\alpha x$ and therefore $cx'=x$. Similarly, since $|\beta'|>|\beta|$ and $\beta x=\beta ' x'$, there exists $d\in \lang$ (with $|d|>0$) such that $\beta'=\beta d$ and $dx'=x$. Joining the previous observations, we obtain that $cx'=dx'$. Suppose that $|c|\neq |d|$ and, without loss of generality,  suppose that $|c|>|d|$. Then, since $cx'=dx'$, there is an $e\in \lang$ (with $|e|>0$) such that $c=de$. Therefore, $dex'=cx'=dx'$ and hence $ex'=x'$. From this last equality, we get that $x'=e^{\infty}$, which is impossible since $q$ is a line path. So, we must have that $|c|=|d|$. Since $cx'=dx'$, we get that $c=d$. This implies that $\alpha'=\alpha c$ and $\beta'=\beta c$, which contradicts the assumption that $(\alpha',\beta')$ is a reduced initial pair. Therefore, the case $|\alpha'|>|\alpha|$ and $|\beta'|>|\beta|$ can not happen.

The case $|\alpha|>|\alpha'|$ and $|\beta|>|\beta'|$ is also not possible, as one can show analogously to what is done in the paragraph above.

Next, suppose that $|\alpha'|>|\alpha|$ and $|\beta'|<|\beta|$. As before, since $\alpha'x'=\alpha x$, there exists $c\in \lang$ (with $|c|>0$) such that $\alpha'=\alpha c$ and $cx'=x$. Similarly, 
there exists $d\in \lang$ (with $|d|>0$) such that $\beta=\beta' d$ and $dx=x'$. Joining these observations, we obtain that $dcx'=dx$ and hence $dcx'=x'$. From this last equality, we get that $x'=(dc)^\infty$, which contradicts the assumption that $q$ is a line path. Therefore, the case $|\alpha'|>|\alpha|$ and $|\beta'|<|\beta|$ is not possible. 

The case $|\alpha'|<|\alpha|$ and $|\beta'|>|\beta|$ is dealt with analogously to the case in the paragraph above, and is also not possible. 

In conclusion, the unique possible case is $|\alpha'|=|\alpha|$ and $|\beta'|=|\beta|$, and so the lemma is proved.

\end{proof}


Our next goal is to define an isomorphism between $\PP_{[q]}$, where $q=aa_1a_2...$ is a line path, and $\ualgshift p_{Z_a}$. We start by defining the relevant map between the modules. 

\begin{definition}\label{aparelho} Let $\osf$ be a subshift and $q=aa_1a_2...$ a line path. We define an $R$-module homomorphism
$$\psi:\PP_{[q]}\rightarrow \ualgshift p_{Z_a}$$
on the basis of $\PP_{[q]}$ as follows: for an element $p\in [q]$, let $(\beta,\alpha)$ be the (unique) reduced initial pair associated with $(p,q)$ and define $\psi(\delta_p)=s_\beta s_\alpha^* p_{Z_a}$. 
\end{definition}

\begin{remark}\label{coelhos} In the setting of the above definition, notice that:
\begin{itemize} 
\item If $|\alpha|=0$, then $\psi(\delta_p)=s_\beta p_{Z_a}$,
\item If $|\beta|=0$, then $\psi(\delta_p)=s_\alpha^* p_{Z_a}$,
\item If $|\alpha|=0=|\beta|$ (which is the case if and only if $q=p$), then $\psi(\delta_q)=p_{Z_a}$. 
\end{itemize}
\end{remark}

We show next that $\psi$ defined above is an isomorphism.

\begin{theorem}\label{eletrico} Let $R$ be a field, $\osf$ a subshift, $q=aa_1a_2...\in \osf$ a line path, and $L_y:\ualgshift p_{Z_a}\rightarrow \ualgshift p_{Z_a}$ the liner map defined by $L_y(\mu)=y\mu$, for each $\mu \in \ualgshift p_{Z_a}$.
Then, the $R$-module homomorphism $\psi:\PP_{[q]}\rightarrow \ualgshift p_{Z_{a}}$ of Definition~\ref{aparelho} is an isomorphism and moreover, for each $y\in\ualgshift$, the diagram below commutes
\begin{displaymath}
\xymatrix{
\PP_{[q]} \ar[r]^{\psi} \ar[d]_{\pi(y)} &
\ualgshift p_{Z_a} \ar[d]^{L_y} \\
\PP_{[q]} \ar[r]_{\psi} & \ualgshift p_{Z_a} }
\end{displaymath}
where $\pi$ is the homomorphism of Proposition~\ref{representationtheorem}. 
\end{theorem}

\begin{proof}  First we show that $\psi$ is surjective. For this notice that, by the second item of Lemma~\ref{linearspan}, it is enough to show that $s_cs_d^*p_{Z_a}\in Im(\psi)$ for each $c,d\in \lang$ with $c_{|c|}\neq d_{|d|}$ and $s_cs_d^* p_{Z_a}\neq 0$. We divide the proof is a few cases.

Suppose that $|c|=0$ and $|d|>0$. Write $d=d_1...d_m\in \lang$, where $d_i\in \alf$ for each $i$. Since $s_d^*p_{Z_a}=0$ in case $d_1\neq a$, we conclude that $d_1=a$. As $q=aa_1a_2...$ is a line path, we obtain that $q=dz$, with $z\in \osf$. Then, $z\in [q]$ and the reduced initial pair associated with $(z,q)$ is the pair $(w,d)$, where $w$ is the empty word. From Remark~\ref{coelhos}, we get that $\psi(\delta_z)=s_d^*p_{Z_a}$.

Next, suppose that $|d|=0$ and $|c|>0$. Write $c=c_1...c_n$,
where $c_i\in \alf$ for each $i$. Notice that $s_cp_{Z_a}=s_c p_{F_c}p_{Z_a}$, since $s_c=s_cs_c^*s_c=s_c p_{F_c}$. We may assume that $F_c\cap Z_a\neq\emptyset$, since otherwise $s_c p_{F_c}p_{Z_a}=0$. Let $z\in F_c\cap Z_a$. Then $cz\in \osf$, since $z\in F_c$, and $z=q$,  since $z\in Z_a=\{q\}$. Let $y=cz=cq$. Then, $y\in [q]$ and $(c,w)$ is the reduced initial pair associated to $(y,q)$. By Remark~\ref{coelhos}, we conclude that $\psi(\delta_y)=s_c p_{Z_a}$.

Now, suppose that $|c|>0$ and $|d|>0$. Observe that $$s_cs_d^*p_{Z_a}=s_cs_c^*s_cs_d^*s_ds_d^*p_{Z_a}=s_cp_{F_c}p_{F_d}s_d^*p_{Z_a}=s_cp_{F_c}p_{F_d}p_{r(Z_a,d)}s_d^*,$$ where the last equality follows from \cite[Lemma 3.3]{reductiontheoremofsubshifts}.
Since $s_cs_d^*p_{Z_a}\neq 0$, we have that $F_c\cap F_d\cap r(Z_a, d)\neq \emptyset.$ Let $z\in F_c\cap F_d\cap r(Z_a,d)$. Then $cz\in \osf$, $dz\in \osf$, and $dz\in Z_a$. Since $q=aa_1a_2...$ is a line path we have that $Z_a=\{q\}$. So, $dz=q$ and therefore $cz\thicksim dz=q$. The pair $(c,d)$ is the reduced initial pair associated with $(cz,q)$, since $c_{|c|}\neq d_{|d|}$. From the definition of $\psi$, we get that $\psi(\delta_{cz})=s_cs_d^*p_{Z_a}$.

The final case, in which $|d|=0$ and $|c|=0$, follows directly from Remark~\ref{coelhos}. This concludes the proof that $\psi$ is surjective.

We now proceed to show that the diagram in the statement of the proposition commutes, for each $y\in \ualgshift$. Since $S:=\{p_A, s_e, s_e^*:A\in \TCB, e\in \alf\}$ is a generating set, and the maps $L:\ualgshift \rightarrow End_R(\ualgshift p_{Z_a})$, given by $L(y)=L_y$, and $\pi$ are algebra homomorphisms, it is enough to prove the commutativity of the diagram for the elements in S.

We begin with an element of the form $p_A$, for $A\in \TCB$. Let $x\in [q]$ and $(c,d)$ be the reduced initial pair associated with $(x,q)$. Then, $\psi(\delta_x)=s_cs_d^* p_{Z_a}$ and hence $$L_{p_A}(\psi(\delta_x))=p_A s_c s_d^* p_{Z_a}.$$ Let $\Phi$ be the isomorphism of Theorem~\ref{thm:set-theoretic-partial-action}. It follows by direct calculations that $\Phi(p_As_cs_d^* p_{Z_a})\neq 0$ if, and only if, $x\in A$. Therefore, $p_As_cs_d^* p_{Z_a}\neq 0$ if, and only if, $x\in A$. On the other hand, $$\psi(\pi(p_A)\delta_x)=[x\in A]\psi(\delta_x)=[x\in A]s_cs_d^* p_{Z_a}.$$
So, in case $x\notin A$, we get that
$$L_{p_A}(\psi(\delta_x))=0=\psi(\pi(p_A)(\delta_x)).$$ In case $x\in A$, notice that $0\neq p_As_cs_d^*p_{Z_a}=s_cp_{r(A,c)}s_d^* p_{Z_a}$. From the first item of Lemma \ref{linearspan}, we conclude that 
$$s_cp_{r(A,c)}s_d^* p_{Z_a}=s_cs_d^* p_{Z_a}.$$
Therefore, in this case, we have that 
$$L_{p_A}(\psi(\delta_x))=s_cs_d^*p_{Z_a}=\psi(\pi(p_A)(\delta_x)).$$
We have proved that $L_{p_A}(\psi(\delta_x))=\psi(\pi(p_A)(\delta_x))$ for each element $\delta_x$ in the basis of $\PP_{[q]}$, which implies that $L_{p_A}\circ \psi=\psi\circ \pi(p_A)$.

Next, we consider an element of the form $y=s_e$, where $e\in \alf$. Let $x\in [q]$ and $(c,d)$ be the reduced initial pair associated with $(x,q)$. Suppose that  $ex\in \osf$. Then, $(ec,d)$ is the reduced initial pair associated with $(ex,q)$ and therefore 
$$\psi(\pi(s_e)(\delta_x))=\psi(\delta_{ex})=s_{ec}s_d^*p_{Z_a}=s_e s_c s_d^*p_{Z_a}=L_{s_e}(\psi(\delta_x)).$$
Suppose now that $ex\notin \osf$. Then, $s_e s_cs_d^* p_{Z_a}=0$ (since $\Phi(s_e s_cs_d^* p_{Z_a})=0$, where $\Phi$ is the isomorphism of Theorem~\ref{thm:set-theoretic-partial-action}). Moreover, since $ex\notin \osf$, we have that $\pi(s_e)(\delta_x)=0$. Therefore, in this case 
$$L_{s_e}(\psi(\delta_x))=s_es_cs_d^*p_{Z_a}=0=\psi(\pi(s_e)(\delta_x)).$$
We have proved that $L_{s_e}(\psi(\delta_x))=\psi(\pi(\delta_x))$ for each $x\in [q]$. Hence, $L_{s_e}\circ \psi=\psi\circ \pi(s_e)$.

The final type of element we need to consider is an element of the form $y=s_e^*$, where $e\in \alf$. Let $x=x_1x_2...\in [q]$ and $(c,d)$ be the reduced initial pair associated with $(x,q)$. Analogously to the previous cases, we obtain that $\Phi(s_e^* s_c s_d^*p_{Z_a})=0$ if, and only if, $x_1=e$, where $\Phi$ is the isomorphism of Theorem~\ref{thm:set-theoretic-partial-action}. So, $s_e^*s_cs_dp_{Z_a}\neq 0$ if, and only if, $x_1=e$. 
Therefore, in case $x_1\neq e$, we have that $$\psi(\pi(s_e^*)(\delta_x))=0=L_{s_e^*}(\psi(\delta_x)).$$
Suppose that $x_1=e$. If $|c|=0$ then the reduced initial pair associated with $(x_2x_3...,q)$ is the pair $(w, de)$ and hence
$$\psi(\pi(s_e^*)(\delta_x))=\psi(\delta_{x_2x_3...})=s_{de}^*p_{Z_a}=s_e^*s_d^*p_{Z_a}=L_{s_e^*}(\psi(\delta_x)).$$
If $|c|>0$, then the first letter of $c$ is $e$, so that $c=e\widetilde{c}$. Therefore, 
$$L_{s_e^*}(\psi(\delta_x))=s_e^*s_c s_d^* p_{Z_a}=s_e^*s_e s_{\widetilde{c}}s_d^* p_{Z_a}=p_{F_e}s_{\widetilde{c}}s_d^*p_{Z_a}=$$
$$=s_{\widetilde{c}} p_{r(F_e,\widetilde{c})}s_d^*p_{Z_a}=s_{\widetilde{c}}s_d p_{Z_a}=\psi(\pi(s_e^*)(\delta_x)),$$ where the second to last equality follows from the first item of Lemma~\ref{linearspan}. We have proved that $L_{s_e^*}(\psi(\delta_x))=\psi(\pi(s_e^*)(\delta_x))$ for each $x\in [q]$. Therefore, $\psi\circ \pi(s_e^*)=L_{s_e^*}\circ \psi$.

To finish the proof we have to show that $\psi$ is injective. For this, we show that $\text{Ker}(\psi)$ is invariant under $\pi$ and use Theorem~\ref{invariant submodules}. Let $y\in \ualgshift$ and $z\in \text{Ker}(\psi)$. Then, $$\psi(\pi(y)(z))=\pi(y)(\psi(z))=\pi(y)(0)=0.$$
So, $\pi(y)(z)\in \text{Ker}(\psi)$ for every $y\in \ualgshift$, that is, $\text{Ker}(\psi)$ is invariant under $\pi$. By Theorem~\ref{invariant submodules}, since $R$ is a field, we conclude that $\text{Ker}(\psi)=0$, as desired.

\end{proof}

The above result allows us to show that the left ideals $\ualgshift p_{Z_{a}}$, associated with line paths, are minimal, see below.

\begin{corollary}\label{madruga}  Let $R$ be a field, $\osf$ a subshift, and $q=aa_1a_2...\in \osf$ a line path. Then, $\ualgshift p_{Z_a}$ is a minimal left ideal of $\ualgshift$.
\end{corollary}
\begin{proof}

In this proof, we use the same maps of Theorem~\ref{eletrico}.

Let $J$ be a left ideal of $\ualgshift$ contained in $\ualgshift p_{Z_a}$. Then, $\psi^{-1}(J)$ is an invariant submodule of $\PP_{[q]}$. Indeed, let $y\in \ualgshift$. By the commutativity of the diagram in Theorem~\ref{eletrico}, we have that \[\pi(y)\left(\psi^{-1}(J)\right) = \psi^{-1}\left(L_y(J)\right) \subseteq \psi^{-1}(J),
\]
where the containment above follows from the fact that $J$ is a left ideal of $\ualgshift$. It is clear that $\psi^{-1}(J)$ is a submodule.
So, by Theorem~\ref{invariant submodules}, we conclude that $\psi^{-1}(J)=\PP_{[q]}$ or $\psi^{-1}(J)=0$. Since $\psi$ is an isomorphism, this implies that $J=\ualgshift p_{Z_a}$ or $J=0$, as desired.
\end{proof}

\begin{remark}
    In \cite{Chen}, minimal left ideals of Leavitt path algebras associated with sinks are also studied. But, in a subshift algebra, there are no sinks, in the sense that the initial data only deals with infinite paths. Therefore, there is no clear substitute for the left ideals associated with sinks studied in \cite{Chen}. In fact, we believe that in the context of subshift algebras, these ideals do not exist, and intend to obtain a complete characterization of all left minimal ideals in a follow-up paper.
\end{remark}

\bibliographystyle{abbrv}
\bibliography{ref}

\end{document}